\documentclass[12pt,reqno]{amsart}
\usepackage{amsthm}
\usepackage{booktabs}

\usepackage{amssymb}
\usepackage{graphics}   
\usepackage{tikz}
\usetikzlibrary{shapes,backgrounds,calc}
\usepackage{latexsym}
\usepackage{multicol}
\usepackage{verbatim,enumerate}
\usepackage{accents}
\usepackage{cite}
\usepackage{enumitem}
\usepackage{array}
\usepackage{mathtools}
\usepackage{dsfont}

\usepackage[colorlinks=true, linkcolor=blue, citecolor=blue, urlcolor=blue]{hyperref}
\usepackage{hyperref}
\usepackage{amsmath, amscd,url}

\usepackage{setspace}

\usepackage{pstricks}

\advance\textwidth by 1.5in 
\usepackage{geometry}
\newtheorem{thm}{Theorem}[section]
\newtheorem{cor}[thm]{Corollary}
\newtheorem{lem}[thm]{Lemma}
\newtheorem{prop}[thm]{Proposition}
\theoremstyle{definition}

\newtheorem{defn}[thm]{Definition}
\newtheorem{example}[thm]{Example}
\newtheorem{rem}[thm]{Remark}

\numberwithin{equation}{subsection}

\newenvironment{pf}{\proof}{\endproof}
\newcounter{cnt}
 \makeatletter
\def\mydggeometry{\makeatletter\dg@YGRID=1\dg@XGRID=20\unitlength=0.003pt\makeatother}
\makeatother \theoremstyle{remark}

\numberwithin{equation}{section}
\let\bwdg\bigwedge
\def\bigwedge{{\textstyle\bwdg}}

\newcommand{\thmref}[1]{Theorem~\ref{#1}}

\newcommand{\propref}[1]{Proposition~\ref{#1}}

\newcommand{\nc}{\newcommand}
\newcommand{\rnc}{\renewcommand}

\nc{\cal}{\mathcal} \nc{\goth}{\mathfrak} \rnc{\bold}{\mathbf}

\nc\bomega{{\mbox{\boldmath $\omega$}}} \nc\bpsi{{\mbox{\boldmath $\Psi$}}}
 \nc\balpha{{\mbox{\boldmath $\alpha$}}}
 \nc\bpi{{\mbox{\boldmath $\pi$}}}
 \nc\bvpi{{\mbox{\boldmath $\varpi$}}}
\nc\chara{\operatorname{ch}}

  \nc\bxi{{\mbox{\boldmath $\xi$}}}
\nc\bmu{{\mbox{\boldmath $\mu$}}} \nc\bcN{{\mbox{\boldmath $\cal{N}$}}} \nc\bcm{{\mbox{\boldmath $\cal{M}$}}} \nc\blambda{{\mbox{\boldmath
$\lambda$}}}\nc\bnu{{\mbox{\boldmath $\nu$}}}

\makeatletter
\def\section{\def\@secnumfont{\mdseries}\@startsection{section}{1}%
  \z@{.7\linespacing\@plus\linespacing}{.5\linespacing}%
  {\normalfont\scshape\centering}}
\def\subsection{\def\@secnumfont{\bfseries}\@startsection{subsection}{2}%
  {\parindent}{.5\linespacing\@plus.7\linespacing}{-.5em}%
  {\normalfont\bfseries}}
\makeatother

 \nc{\Hom}{\operatorname{Hom}}
  \nc{\mode}{\operatorname{mod}}
\nc{\End}{\operatorname{End}} \nc{\wh}[1]{\widehat{#1}} \nc{\Ext}{\operatorname{Ext}} \nc{\ch}{\text{ch}} \nc{\ev}{\operatorname{ev}}
\nc{\Ob}{\operatorname{Ob}} \nc{\soc}{\operatorname{soc}} \nc{\rad}{\operatorname{rad}} \nc{\head}{\operatorname{head}}

\def\det{\operatorname{det}}

 \nc{\Cal}{\cal} \nc{\Xp}[1]{X^+(#1)} \nc{\Xm}[1]{X^-(#1)}
\nc{\on}{\operatorname} \nc{\Z}{{\bold Z}} \nc{\J}{{\cal J}}  \nc{\Q}{{\bold Q}}

\nc{\N}{{\bold N}}  \nc\boa{\bold a} \nc\bob{\bold b} \nc\boc{\bold c} \nc\bod{\bold d} \nc\boe{\bold e} \nc\bof{\bold f} \nc\bog{\bold g}
\nc\boh{\bold h} \nc\boi{\bold i} \nc\boj{\bold j} \nc\bok{\bold k} \nc\bol{\bold l} \nc\bom{\bold m} \nc\bon{\mathbb n} \nc\boo{\bold o}
\nc\bop{\bold p} \nc\boq{\bold q} \nc\bor{\bold r} \nc\bos{\bold s} \nc\boT{\bold t} \nc\boF{\bold F} \nc\bou{\bold u} \nc\bov{\bold v}
\nc\bow{\bold w} \nc\boz{\bold z}\nc\ba{\bold A} \nc\bb{\bold B} \nc\bc{\mathbb C} \nc\bd{\bold D} \nc\be{\bold E} \nc\bg{\bold
G} \nc\bh{\bold H} \nc\bi{\bold I} \nc\bj{\bold J} \nc\bk{\bold K} \nc\bl{\bold L} \nc\bm{\bold M} \nc\bn{\mathbb N} \nc\bo{\bold O} \nc\bp{\bold
P} \nc\bq{\bold Q} \nc\br{\bold R} \nc\bs{\bold S} \nc\bt{\bold T} \nc\bu{\bold U} \nc\bv{\bold V} \nc\bw{\bold W} \nc\bz{\mathbb Z} \nc\bx{\bold
x} \nc\KR{\bold{KR}} \nc\rk{\bold{rk}} \nc\het{\text{ht }}

\nc\toa{\tilde a} \nc\tob{\tilde b} \nc\toc{\tilde c} \nc\tod{\tilde d} \nc\toe{\tilde e} \nc\tof{\tilde f} \nc\tog{\tilde g} \nc\toh{\tilde h}
\nc\toi{\tilde i} \nc\toj{\tilde j} \nc\tok{\tilde k} \nc\tol{\tilde l} \nc\tom{\tilde m} \nc\ton{\tilde n} \nc\too{\tilde o} \nc\toq{\tilde q}
\nc\tor{\tilde r} \nc\tos{\tilde s} \nc\toT{\tilde t} \nc\tou{\tilde u} \nc\tov{\tilde v} \nc\tow{\tilde w} \nc\toz{\tilde z} \nc\woi{w_{\omega_i}}

\newcommand\wrapped[1]%
{\renewcommand\arraystretch{1}%
	\begin{array}{@{}l@{}}#1\end{array}%
}

\ifTUTeX
  \usepackage{fontspec}
\else
  \usepackage[T1]{fontenc}
  \usepackage[utf8]{inputenc} 
  \DeclareUnicodeCharacter{200B}{{\hskip 0pt}}
\fi

\begin{document}
	
	\title{Zero-sum Inverse Realization and Property~(P) under
Join Operations}

   \author{G. Arunkumar}
	\address{Indian Institute of Technology Madras, Chennai, India.}
	\email{garunkumar@iitm.ac.in}
    
    \author{Anubhab Pahari}
	\address{Indian Institute of Technology Madras, Chennai, India.}
	\email{ma22d012@smail.iitm.ac.in, anubhabpahari@gmail.com}
    
	\author{Puja Samanta}
	\address{Indian Institute of Technology Madras, Chennai, India.}
	\email{ma23d004@smail.iitm.ac.in, pujasamanta1999@gmail.com}




\subjclass[2020]{05C50, 05C76, 05C38, 15A18.}

\begin{abstract}
We introduce zero-sum inverse realization of property (P) of a graph $G$, obtained by imposing an additional condition
\(
\mathbf{1}^{\top}A^{-1}\mathbf{1}=0,
\)
on a matrix $A\in S(G)$ realizing property (P), where $\mathbf{1}$ is the all-ones vector.
We prove that every graph of order at least three having property (P) admits such a zero-sum inverse realization. As applications, we prove that property~(P) is preserved under the join of two graphs of order at least $3$ and, more generally, under the $H$-join of a family of graphs of order at least $3$, where $H$ is arbitrary. Consequently, we obtain sufficient conditions for cographs and lexicographic products of graphs to possess property~(P). Throughout the paper, many examples are given.
\end{abstract}

\maketitle

\medskip

\noindent\textbf{Keywords:}
Property (P); P-vertex; indefinite quadratic form; isotropic vector; graph join; $H$-join.

\section{Introduction}

Standard graph-theoretic terminology follows \cite{west}. Let $A=(A_{ij})$ be a $n\times n$ real symmetric matrix. The support graph of $A$, denoted by $G(A)$, is the simple graph whose vertex set is $\{1,\dots,n\}$ and edge set is $\{ij \ | \ i \ne j \ \text{and} \ A_{ij}\ne 0\}$. We denote the transpose of $A$ by $A^\top$. For a graph $G$, let $$ S(G)=\{A\in\mathbb{R}^{n\times n}| \ A^\top=A\ \text{and}\ G(A)=G\}.$$  If $G$ is an acyclic graph, then the matrices in $S(G)$ are called \emph{acyclic matrices}.

For a subset $\alpha$ of the vertex set of $G$, we denote by $A(\alpha)$ the principal submatrix obtained by deleting the rows and columns indexed by $\alpha$; in particular, we denote by $A(i)$ the $(n-1)\times (n-1)$ principal submatrix of $A$ obtained by deleting the row and column indexed with $i$. 
If $m_A(\lambda)$ denotes the multiplicity of the eigenvalue $\lambda$ of $A$, from Cauchy's Interlacing Theorem \cite[Theorem 4.3.17]{inter}, we know that $$ -1 \le m_{A(i)}(\lambda) - m_{A}(\lambda) \le 1.$$
A vertex $i$ of the support graph of $A$ is called \emph{Parter, downer, or neutral} vertex of $A$ with respect to $\lambda$, depending on whether $m_{A(i)}(\lambda) - m_{A}(\lambda) = 1, -1,\ \text{or} \ 0,$ respectively \cite{Fiedler,hermitian}. In particular, when $m_{A(i)}(0) - m_{A}(0)=1$, $i$ is called \emph{P-vertex} of $A$ \cite{kim1}. The number of P-vertices of $A$ is denoted by $P_{\nu}(A)$. A graph $G$ of order $n$ is said to have \emph{full P-vertices} if there exists $A \in S(G)$ with $P_\nu(A)=n$. 


The study of P-vertices began with the work of Parter \cite{Parter}, who related graph structure, particularly for trees, to eigenvalue multiplicities of real symmetric matrices, with further developments in \cite{wiener, hermitian, kim1, kim2}. Johnson and Sutton \cite{hermitian} showed that a singular acyclic matrix of order $n$ has at most $n-2$ P-vertices, while Kim and Shader \cite{kim1} proved that even paths admit matrices with all vertices as P-vertices, whereas every matrix associated with an odd path has at most $n-1$ P-vertices. Extensive work on acyclic graphs has since established bounds and structural characterizations for the existence of matrices with many P-vertices \cite{hermitian,kim1,fons1,fons2,fons3,fons4,fons5,fons6,du1,du2,du3,du4,du5,cruz, kim2}. In particular, trees admitting matrices with the full P-vertex property and prescribed nullity have been characterized in \cite{fons2,fons5,du2}.



Beyond acyclic graphs, cycle graphs are also known to have the full P-vertex property \cite[Theorem 4.2]{fons1}. 
In 2025, property~(P) was introduced by Howlader et al., as a reformulation of the full P-vertex property for nonsingular symmetric matrices \cite[Definition~1.1]{Howlader}.

\begin{defn}\label{def:P}
A graph \(G\) on \(n\) vertices has \emph{property~(P)} if there exists a nonsingular matrix
\(A\in S(G)\) such that
\(
P_{\nu}(A)=n.
\)
\end{defn}
By Jacobi's identity
$(A^{-1})_{ii}=\det A(i)/\det A$, so
Definition~\ref{def:P} admits the equivalent analytic form that will be used
throughout:
\begin{equation*}\label{eq:hollow}
   G \text{ has property (P)} \iff \exists\, A\in S(G) \text{ nonsingular
   with } \operatorname{diag}(A^{-1})=\mathbf{0}.
\end{equation*}
They studied the property~(P) of unicyclic graphs and obtained a complete characterization of connected unicyclic graphs having property~(P) \cite[Theorem~6.1]{Howlader}.
More recently, Sharma and Panda showed that the existence of a perfect matching guarantees property~(P) in bipartite graphs \cite[Theorem~3.1]{sharma}. For trees and unicyclic bipartite graphs, they established that property~(P) is equivalent to the existence of a perfect matching \cite[Theorems~3.3 and~3.8]{sharma}. They also characterized several additional graph classes and showed that connecting two graphs with property~(P) by a single edge preserves property~(P) \cite[Theorem~5.3]{sharma}. Beyond this result, the behaviour of property~(P) under graph operations has remained largely unexplored. 
Extending this preservation result to more general graph operations is far from straightforward. Unlike the addition of a single edge, the join operation simultaneously introduces all possible edges between two vertex sets. Moreover, property~(P) is characterized by the diagonal of the inverse of a realizing matrix, while the inverse of a block matrix depends globally on all of its blocks. Consequently, the effect of such a dense graph operation cannot be understood through local edge
additions alone.

Extending preservation results from a single-edge addition to the join
and, more generally, the \(H\)-join is the main goal of this paper. 
To accomplish this, we introduce the notion of a \emph{zero-sum inverse realization} of property~(P): a nonsingular matrix \(A\in S(G)\) such that every vertex of \(G\) is a P-vertex of \(A\) and
\(
\mathbf{1}^{\top}A^{-1}\mathbf{1}=0,
\)
where \(\mathbf{1}\) denotes the all-ones vector.
Despite this additional algebraic constraint, we prove that every graph of order at least $3$ having property~(P) admits such a realization (\thmref{general}). The proof relies on a property of nondegenerate indefinite quadratic forms over $\mathbb{R}^n$, where $n$ is the order of the graph.  The essential observation is that every nondegenerate indefinite quadratic form of dimension at least $3$ admits an isotropic vector with no zero coordinates (\propref{prop: nonzero isotropic vector}), which enables the construction of the required realization through diagonal scaling.

The zero-sum inverse realization enables us to prove that property~(P) is preserved under the join of two graphs of order at least \(3\) (\thmref{join}) and, more generally, under the \(H\)-join of a family of pairwise vertex-disjoint graphs of order at least \(3\), where \(H\) is arbitrary (\thmref{thm:H-join}).
These results provide a systematic method for constructing new graphs with property~(P).
 In particular, we obtain a sufficient condition for cographs to possess property~(P), leading to complete multipartite graphs of the form
\(
K_{n_1,n_1,\ldots,n_t,n_t},
\)
and balanced Tur\'an graphs with an even number of parts. We further prove that, whenever \(G\) has property~(P), the lexicographic product \(H[G]\) also has property~(P) for every graph \(H\), providing a general method for constructing further examples.



This article is organized as follows. Section~\ref{sec:zero-sum} establishes the key property of a nondegenerate indefinite quadratic form required for the construction of zero-sum inverse realizations and proves that every graph of order at least \(3\) with property~(P) admits such a realization. Section~\ref{sec:join} establishes the preservation of property~(P) under the join operation. Finally, Section~\ref{sec:H-join} extends this result to the \(H\)-join and presents applications to cographs and lexicographic products.

\section{Zero-Sum Inverse Realization of Property (P)}\label{sec:zero-sum}

We begin by introducing the notion of zero-sum inverse realizations of property~(P).

\begin{defn}\label{def:zero_sum_realization}
Let $G$ be a graph of order $n$ having property (P). A nonsingular matrix
$A\in S(G)$ is called a \emph{zero-sum inverse realization} of property (P) if
\[
P_\nu(A)=n
\quad\text{and}\quad
\mathbf{1}^{\top}A^{-1}\mathbf{1}=0,
\]
where $\mathbf{1}=(1,1,\ldots,1)^{\top} \in \mathbb{R}^n$. Equivalently,
\(
\sum_{i,j=1}^{n}(A^{-1})_{ij}=0.
\)
\end{defn}
The following proposition explains why, throughout the paper, we restrict our attention to graphs of order at least \(3\). Observe first that a graph of order \(1\) cannot have property~(P), since the inverse of every nonsingular matrix of order \(1\) has a nonzero diagonal entry.

\begin{prop}\label{K_2}
A graph of order $2$ does not admit a zero-sum inverse realization of property (P).
\end{prop}
\begin{pf}
Suppose that a graph $G$ of order $2$ admits a zero-sum inverse realization of property (P). Then $G$ is either the complete graph $K_2$ or its complement $\overline{K_2}$. Clearly, it cannot be $\overline{K_2}$ because any nonsingular diagonal matrix must have non-zero diagonal entries, which violates the requirement that every $1 \times 1$ principal submatrix is singular. Therefore, $G=K_2$. Since $G$ admits such a realization, it also has property~(P). Hence any matrix $A\in S(K_2)$ realizing property~(P) must be of the form
\[
A=
\begin{pmatrix}
0 & a\\
a & 0
\end{pmatrix},
\qquad a \in \mathbb{R}\setminus \{0\}.
\]
Then, 
\[
A^{-1}= \begin{pmatrix}
    0 & \frac{1}{a}\\[2mm]
    \frac{1}{a} & 0
\end{pmatrix}.
\]
The sum of the entries of $A^{-1}$ is $\frac{2}{a}$, which cannot be zero for any $a$. This contradicts the assumption. Hence, no graph of order $2$ admits this realization.
\end{pf}
\begin{example}
    The graph $C_3$ admits a zero-sum inverse realization of property (P).  Consider the matrix $A\in S(C_3)$,
   \[
\renewcommand{\arraystretch}{1.4}
A=
\begin{pmatrix}
1 & \frac{1}{2} & \frac{1}{2}\\
\frac{1}{2} & \frac{1}{4} & -\frac{1}{4}\\
\frac{1}{2} & -\frac{1}{4} & \frac{1}{4}
\end{pmatrix}
\]
$A$ is nonsingular and $\det(A(i))=0$ for all $i \in \{1,2,3\}$. Now, the inverse of $A$,
\[
A^{-1}= \begin{pmatrix}
    0 & 1 & 1\\
    1 & 0 & -2\\
    1 & -2 & 0\\
\end{pmatrix}
\]
The sum of all elements in $A^{-1}$ is zero.
\end{example}

We now prove that every graph of order at least \(3\) having property~(P) admits a zero-sum inverse realization of property~(P). The proof relies on a structural property of nondegenerate indefinite quadratic forms. We begin with the following elementary fact concerning polynomials.

\begin{lem}\label{avoidance}
For $d\ge 1$, let $f\in \mathbb{R}[x_1,\dots,x_d]$ be a nonzero polynomial. Then $f$ does not vanish on any nonempty open subset of $\mathbb{R}^d$, where $\mathbb{R}^d$ is equipped with the standard Euclidean topology.
\end{lem}

\begin{pf}
We prove the result by induction on $d$.

For $d=1$, the result is immediate, since a nonzero polynomial in one variable has only finitely many roots in $\mathbb{R}$. Hence it cannot vanish on a nonempty open interval.

Assume now that $d>1$, and suppose that the result holds for polynomials in at most $d-1$ variables. Write $f$ as a polynomial in the variable $x_d$,
\[
f(x_1,\dots,x_d)
=
\sum_{j=0}^{m} c_j(x_1,\dots,x_{d-1})x_d^j,
\]
where $c_j\in \mathbb{R}[x_1,\dots,x_{d-1}]$. Since $f$ is nonzero, at least one coefficient polynomial, say $c_{j_0}$, is nonzero.

Suppose, for a contradiction, that $f$ vanishes on a nonempty open set $U\subseteq \mathbb{R}^d$. Since $U$ is a nonempty open set in the Euclidean topology, there exist a nonempty open set $V\subseteq \mathbb{R}^{d-1}$ and a nonempty open interval $I\subseteq \mathbb{R}$ such that
\(
V\times I\subseteq U.
\)
Therefore,
\[
f(\bar a,x_d)=0
\quad
\text{for all } \bar a\in V \text{ and all } x_d\in I.
\]

Fix an arbitrary point $\bar a=(a_1,\dots,a_{d-1})\in V$. Consider the one-variable polynomial
\[
q(x_d)=f(\bar a,x_d)
=
\sum_{j=0}^{m} c_j(\bar a)x_d^j.
\]
Since $f$ vanishes on $V\times I$, the polynomial $q$ vanishes for all $x_d\in I$. By the one-dimensional case, $q$ must be the zero polynomial. Hence all of its coefficients are zero, that is,
\[
c_j(\bar a)=0
\quad
\text{for all } j=0,\dots,m.
\]
Since $\bar a\in V$ was arbitrary, each coefficient polynomial $c_j$ vanishes on $V$. In particular, $c_{j_0}$ vanishes on the nonempty open set $V\subseteq \mathbb{R}^{d-1}$. This contradicts the induction hypothesis, because $c_{j_0}$ is a nonzero polynomial in $d-1$ variables.

Therefore, $f$ cannot vanish on any nonempty open subset of $\mathbb{R}^d$.
\end{pf}

For a vector $y \in \mathbb{R}^n$ whose coordinates are indexed by $V(G)$, we define its \emph{support} as $\operatorname{supp}(y)=\{v \in V(G) : y_v\ne0\}$. A vector $y$ is said to be \emph{fully supported} if $\operatorname{supp}(y)=V(G)$, meaning every coordinate is nonzero.

Next, let $Q: \mathbb{R}^n \rightarrow \mathbb{R}$ be a real quadratic form. Recall that a vector $x\in\mathbb{R}^n$ is called \emph{isotropic} if $Q(x)=0$. $Q$ is called
\emph{indefinite} if it attains both positive and negative values. If $B$
denotes the symmetric bilinear form associated with $Q$, then $Q$ is called
\emph{nondegenerate} if the radical of $B$ is zero, that is,
\[
\{x\in\mathbb{R}^n:B(x,y)=0 \text{ for all } y\in\mathbb{R}^n\}=\{\mathbf{0}\}.
\]
Equivalently, the symmetric matrix representing $Q$ is nonsingular.

To construct such a realization, we require a fully supported isotropic vector. The following proposition provides such a vector.
\begin{prop}\label{prop: nonzero isotropic vector}
Let $Q$ be a nondegenerate indefinite quadratic form on $\mathbb{R}^n$ with
$n\ge3$. Then $Q$ admits a fully supported isotropic vector $x \in \mathbb{R}^n$.
\end{prop}

\begin{pf}
Write $$x=\begin{pmatrix}x'\\x_n\end{pmatrix}\in\mathbb{R}^n,\  \text{where}
\
x'=(x_1,\dots,x_{n-1})^\top\in\mathbb{R}^{n-1}.
$$
Let $B$ denote the symmetric bilinear form associated with $Q$, and define
\[
\phi(x')=Q\begin{pmatrix}x'\\0\end{pmatrix},\qquad
\ell(x')=B\left(e_n,\begin{pmatrix}x'\\0\end{pmatrix}\right),\qquad
c=Q(e_n).
\]
Then
\begin{equation}\label{eq:Q-decomposition}
Q\begin{pmatrix}x'\\x_n\end{pmatrix}
=
\phi(x')+2\ell(x')x_n+cx_n^2.
\end{equation}

We first show that $\phi\not\equiv 0$. Suppose, to the contrary, that
$\phi\equiv 0$. Then $Q$ vanishes on the hyperplane
\(
H=\operatorname{span}\{e_1,\dots,e_{n-1}\}.
\)
If $v,w\in H$, then
\[
Q(v+w)=Q(v)+Q(w)+2B(v,w).
\]
Since $Q$ vanishes on $H$, we get $B(v,w)=0$ for all $v,w\in H$. Hence
\[
H\subseteq H^\perp,
\]
where
\(
H^\perp=\{x\in\mathbb{R}^n:B(x,h)=0\text{ for all }h\in H\}
\)
denotes the orthogonal complement of $H$ with respect to $B$.
Since $Q$ is nondegenerate, the associated bilinear form $B$ is nondegenerate. Hence
\[
\dim H^\perp=n-\dim H=1.
\]
Thus
\[
n-1=\dim H\le \dim H^\perp=1,
\]
which is impossible because $n\ge 3$. Therefore, $\phi$ is a nonzero polynomial.

We now consider two cases.

\medskip

\noindent\textbf{Case 1: $c\neq 0$.}
In this case, equation \eqref{eq:Q-decomposition} can be rewritten as
\[
Q\begin{pmatrix}x'\\x_n\end{pmatrix}
=
c\left(x_n+\frac{\ell(x')}{c}\right)^2
-\frac{1}{c}\Delta(x'),
\]
where
\[
\Delta(x')=\ell(x')^2-c\phi(x').
\]

We next show that $\Delta(x')>0$ for some
$x'\in\mathbb R^{n-1}$. Since $Q$ is indefinite, it
attains a value of sign opposite to $c$. If $c>0$, choose
$u=(u',u_n)^\top\in\mathbb{R}^n$ such that $Q(u)<0$. Then
\[
Q(u)=
c\left(u_n+\frac{\ell(u')}{c}\right)^2
-\frac{1}{c}\Delta(u').
\]
Since the first term on the right-hand side is nonnegative, we must have
\(
-\frac{1}{c}\Delta(u')<0.
\)
As $c>0$, this gives $\Delta(u')>0$.
Similarly, if $c<0$, choose $u\in\mathbb{R}^n$ such that $Q(u)>0$. Then
\(
c\left(u_n+\frac{\ell(u')}{c}\right)^2\le 0,
\)
and hence the equality above forces $\Delta(u')>0$.

Therefore
\(
U=\{x'\in\mathbb{R}^{n-1}:\Delta(x')>0\}
\)
is a nonempty open subset of $\mathbb{R}^{n-1}$.
Consider the polynomial
\(
F(x')=\phi(x')x_1\cdots x_{n-1}.
\)
Since $\phi\not\equiv 0$, the polynomial $F$ is nonzero. By
Lemma~\ref{avoidance}, $F$ cannot vanish on the nonempty open set $U$.
Therefore, there exists
\(
x^\ast=(x_1^\ast,\dots,x_{n-1}^\ast)^\top\in U
\)
such that
\(
F(x^\ast)\neq 0.
\)
Thus
\[
\phi(x^\ast)\neq 0
\quad\text{and}\quad
x_i^\ast\neq 0
\quad\text{for all } i=1,\dots,n-1.
\]

Fix \(x'=x^\ast\). By \eqref{eq:Q-decomposition}, treating the last coordinate as a variable
\(t\in\mathbb{R}\), we obtain the quadratic equation
\[
ct^2+2\ell(x^\ast)t+\phi(x^\ast)=0.
\]
Its discriminant is
\[
4\ell(x^\ast)^2-4c\phi(x^\ast)=4\Delta(x^\ast)>0.
\]
Hence it has two real roots. Moreover, since $\phi(x^\ast)\neq 0$, neither
root is zero. Choosing $x_n$ to be either root gives
\[
Q\begin{pmatrix}x^\ast\\x_n\end{pmatrix}=0,
\]
and all coordinates of this vector are nonzero.

\medskip

\noindent\textbf{Case 2: $c=0$.}
In this case, equation \eqref{eq:Q-decomposition} reduces to
\[
Q\begin{pmatrix}x'\\x_n\end{pmatrix}
=
\phi(x')+2\ell(x')x_n.
\]

We first show that $\ell$ is not identically zero. Indeed,
\[
\ell(x')=
B\left(e_n,\begin{pmatrix}x'\\0\end{pmatrix}\right)
=
\sum_{j=1}^{n-1}B(e_n,e_j)x_j.
\]
If $\ell\equiv 0$, then
\[
B(e_n,e_j)=0
\quad\text{for all } 1\le j\le n-1.
\]
Since $c=0$, we also have
\[
B(e_n,e_n)=Q(e_n)=0.
\]
Therefore,
\[
B(e_n,e_j)=0
\quad\text{for all } 1\le j\le n.
\]
Thus, $e_n$ is orthogonal to all vectors in $\mathbb{R}^n$, which means that
$e_n$ lies in the radical of $B$. This contradicts the nondegeneracy of $Q$.
Hence $\ell\not\equiv 0$.

Consider the polynomial
\[
G(x')=\phi(x')\ell(x')x_1\cdots x_{n-1}.
\]
Observe that $G\not\equiv 0$. By Lemma~\ref{avoidance}, there exists
$x^\ast\in\mathbb{R}^{n-1}$ such that
\[
\phi(x^\ast)\neq 0,\qquad
\ell(x^\ast)\neq 0,\qquad
x_1^\ast\cdots x_{n-1}^\ast\neq 0.
\]
Set
\[
x_n=-\frac{\phi(x^\ast)}{2\ell(x^\ast)}.
\]
Then $x_n\neq 0$, and
\[
Q\begin{pmatrix}x^\ast\\x_n\end{pmatrix}
=
\phi(x^\ast)+2\ell(x^\ast)
\left(-\frac{\phi(x^\ast)}{2\ell(x^\ast)}\right)
=0.
\]
Therefore,
\(
\begin{pmatrix}x^\ast\\x_n\end{pmatrix}
\)
is an isotropic vector for $Q$ whose coordinates are all nonzero.
\end{pf}
We now apply Proposition~\ref{prop: nonzero isotropic vector} to the quadratic form induced by the inverse of a matrix realizing property~(P).

\begin{lem}\label{lem:isotropic_inverse}
If a matrix $A \in S(G)$ of order $n \ge 3$ realizes property~(P), then the quadratic form $Q_A(x) = x^\top A^{-1} x$ admits a fully supported isotropic vector.
\end{lem}

\begin{proof}
 Because $A$ realizes property~(P), it is nonsingular and satisfies $\det A(i) = 0$ for all $i$. Define a quadratic form
\[
Q_A:\mathbb{R}^n\longrightarrow \mathbb{R},
\qquad
Q_A(x)=x^\top A^{-1}x.
\]
Set $M=A^{-1}$,
since
\[
M_{ii}=\frac{\det A(i)}{\det A}=0
\quad\text{for all } i=1,\dots,n,
\]
the matrix $M$ has zero diagonal, and therefore
\[
Q_A(x)=\sum_{i\neq j}M_{ij}x_ix_j
=
2\sum_{1\le i<j\le n}M_{ij}x_ix_j.
\]
We first show that $Q_A$ is a nondegenerate indefinite quadratic form. Nondegeneracy
follows from the nonsingularity of $M$. Since $M$ is nonsingular, it cannot be the zero matrix. As its diagonal entries are all zero, there exist distinct
indices $i$ and $j$ such that $M_{ij}\neq 0$. Then
\[
Q_A(e_i+e_j)=2M_{ij},
\qquad
Q_A(e_i-e_j)=-2M_{ij}.
\]
Thus, $Q_A$ attains both positive and negative values. Therefore, $Q_A$ is indefinite.

Since $Q_A$ is a nondegenerate indefinite quadratic form on $\mathbb{R}^n$ with
$n\ge3$, Proposition~\ref{prop: nonzero isotropic vector} yields an isotropic vector 
\(
x^\ast=(x_1^\ast,\dots,x_n^\ast)^\top\in\mathbb{R}^n
\)
such that
\(
x_i^\ast\neq 0
\text{ for all } i=1,\dots,n.
\)
\end{proof}

We now use Lemma~\ref{lem:isotropic_inverse} to prove the main result of this section.
\begin{thm}\label{general}
Let $G$ be a graph of order $n\ge 3$. If $G$ has property~(P), then $G$ admits a zero-sum inverse realization of property~(P).
\end{thm}

\begin{pf}
Since $G$ has property~(P), there exists a nonsingular matrix $A \in S(G)$ that satisfies $\det A(i)=0$ for all $i=1,\dots,n$. By Lemma~\ref{lem:isotropic_inverse}, the quadratic form $Q_A(x) = x^\top A^{-1} x$ admits a fully supported isotropic vector $x^\ast = (x_1^\ast, \dots, x_n^\ast)^\top \in \mathbb{R}^n$.

Define the scaling matrix
\[
D=\operatorname{diag}\left(
\frac{1}{x_1^\ast},\frac{1}{x_2^\ast},\dots,\frac{1}{x_n^\ast}
\right),
\qquad
B=DAD.
\]
Since $D$ is diagonal, $B$ is symmetric. Moreover, for $i\neq j$,
\(
B_{ij}
=
\frac{A_{ij}}{x_i^\ast x_j^\ast}.
\)
Because each coordinate of $x^\ast$ is nonzero, we have
\[
B_{ij}\neq 0
\iff
A_{ij}\neq 0
\iff
ij\in E(G).
\]
Therefore $B\in S(G)$.

Furthermore,
\[
\det B
=
\det(DAD)
=
\det(D)^2\det A
=
\left(\prod_{k=1}^n\frac{1}{x_k^\ast}\right)^2\det A.
\]
Since each $x_k^\ast\neq 0$ and $\det A\neq 0$, we get $\det B\neq 0$.

Next, for each $i$, we have
\[
B(i)=D(i)A(i)D(i),
\]
where $D(i)$ is the principal submatrix of $D$ obtained by deleting the $i$th
row and the $i$th column. Hence
\[
\det B(i)
=
\det(D(i))^2\det A(i)
=
\left(\prod_{k\neq i}\frac{1}{x_k^\ast}\right)^2\det A(i)
=
0.
\]
Therefore $B$ is nonsingular and $\det B(i)=0$ for all $i$. Hence
\(
P_\nu(B)=n.
\)

It remains to check the zero-sum condition. Since $B=DAD$ and $D$ is
nonsingular,
\[
B^{-1}=(DAD)^{-1}=D^{-1}A^{-1}D^{-1}.
\]
Moreover,
\[
D^{-1}\mathbf 1
=
(x_1^\ast,\dots,x_n^\ast)^\top
=
x^\ast.
\]
Therefore
\[
\begin{aligned}
\mathbf 1^\top B^{-1}\mathbf 1
&=
\mathbf 1^\top D^{-1}A^{-1}D^{-1}\mathbf 1 \\
&=
(D^{-1}\mathbf 1)^\top A^{-1}(D^{-1}\mathbf 1) \\
&=
(x^*)^\top A^{-1} x^*\\
&=
Q_A(x^\ast) \\
&=
0,
\end{aligned}
\]
because $x^\ast$ is isotropic for $Q_A$.
Thus $B\in S(G)$ is a nonsingular matrix with
\[
P_\nu(B)=n
\quad\text{and}\quad
\mathbf{1}^\top B^{-1}\mathbf{1}=0.
\]
Hence $B$ is a zero-sum inverse realization of property~(P) for $G$.
\end{pf}
An immediate consequence of Theorem~\ref{general} is the following.
\begin{cor}
Let $G$ be a graph of order at least $3$ having property~(P).
Then there exists a nonsingular matrix $B\in S(G)$ such that
\(
P_\nu(B)=|V(G)|
\)
and $\mathbf1$ is an isotropic vector of the quadratic form
\(
Q_B(x)=x^\top B^{-1}x.
\)
\end{cor}
The following example gives an explicit construction for the even path \(P_{2m}\), where \(m\ge2\).

\begin{example}
Let \(m\ge2\).
Consider the natural bipartition of $P_{2m}$ given by the partite sets $U=\{1,3,\dots,2m-1\}$ and $V=\{2,4,\dots,2m\}$. With respect to this block vertex ordering, we choose the matrix $A \in S(P_{2m})$ defined by
\[
A=
\begin{pmatrix}
\mathbf0_m&X\\[2mm]
X^\top&\mathbf0_m
\end{pmatrix},
\]
where \(\mathbf0_m\) denotes the \(m\times m\) zero matrix and $X$ is the $m \times m$ lower bidiagonal matrix,
\[
X=
\begin{pmatrix}
1&0&0&\cdots&0\\[2mm]
\dfrac{m}{m-1}&1&0&\cdots&0\\[2mm]
0&-m&1&\ddots&\vdots\\
\vdots&\ddots&\ddots&\ddots&0\\
0&\cdots&0&-m&1
\end{pmatrix}.
\]
Because $X$ is lower triangular with a unit diagonal, it is nonsingular ($\det X = 1$). Consequently, $A$ is nonsingular with inverse
\[
A^{-1}=
\begin{pmatrix}
0&{(X^{\top})}^{-1}\\
X^{-1}&0
\end{pmatrix}.
\]
Since the diagonal of $A^{-1}$ is identically zero, $A$ realizes property~(P). To verify the zero-sum condition, it suffices to show that $\mathbf1_m^\top X^{-1}\mathbf1_m=0$. 

Let
\[
y=
\begin{pmatrix}
1\\[1mm]
-\dfrac{1}{m-1}\\
\vdots\\
-\dfrac{1}{m-1}
\end{pmatrix} \in \mathbb{R}^m.
\]
A direct computation yields $Xy=\mathbf1_m.$ Hence, $X^{-1}\mathbf1_m=y$. Since
\[
\mathbf1_m^\top y
=
1-(m-1)\frac1{m-1}
=0,
\]
we obtain $\mathbf1_m^\top X^{-1}\mathbf1_m=0$, as desired.

For $m=2$, this construction yields the explicit realization for $P_4$ in the natural vertex ordering:
\[
A_{P_4} = 
\begin{pmatrix}
0&1&0&0\\
1&0&2&0\\
0&2&0&1\\
0&0&1&0
\end{pmatrix}
\quad \text{and} \quad
A_{P_4}^{-1} = 
\begin{pmatrix}
0&1&0&-2\\
1&0&0&0\\
0&0&0&1\\
-2&0&1&0
\end{pmatrix}.
\]
As required, $A_{P_4}^{-1}$ has a zero diagonal and the sum of its entries is zero.
\end{example}

 \section{Property (P) for join of graphs}\label{sec:join}
 
In this section, we apply the zero-sum inverse realization established in the previous section to prove that property~(P) is preserved under the join operation.

\begin{defn}[Join of graphs {\cite[Definition~3.3.6]{west}}]
    Let $G_1=(V_1,E_1)$ and $G_2=(V_2,E_2)$ be disjoint graphs. The join of $G_1$ and $G_2$, denoted by $G_1 \vee G_2$, is the graph with vertex set
    \[ 
    V(G_1 \vee G_2)= V_1 \cup V_2, \]
   and edge set
    \[
    E(G_1 \vee G_2)= E(G_1) \cup E(G_2) \cup \bigl\{uv: u \in V(G_1), v \in V(G_2)\bigr\}. 
    \]
\end{defn}

\begin{thm}\label{join}
 Let $G_1$ and $G_2$ be graphs of order at least $3$. If both $G_1$ and $G_2$ have property (P), then their join $G_1 \vee G_2$ has property (P).
\end{thm}
\begin{pf}
Let $n_1=|V(G_1)|$ and $n_2=|V(G_2)|$. Since $G_1$ and $G_2$ have property~(P) and are of order at least $3$, Theorem~\ref{general} implies that each admits a zero-sum inverse realization of property~(P).
Thus, there exist nonsingular matrices
\[
A_1\in S(G_1) \quad \text{and} \quad A_2\in S(G_2)
\]
such that
\[
\det A_1(i)=0 \quad \text{for all } i=1,\dots,n_1,
\]
\[
\det A_2(j)=0 \quad \text{for all } j=1,\dots,n_2,
\]
and their inverses satisfy the zero-sum condition,
\[
\mathbf{1}_{n_1}^\top A_1^{-1}\mathbf{1}_{n_1}=0, \qquad \mathbf{1}_{n_2}^\top A_2^{-1}\mathbf{1}_{n_2}=0.
\]

Define
\[
A = 
\begin{pmatrix}
A_1 & J_{n_1,n_2} \\[2mm]
J_{n_2,n_1} & A_2
\end{pmatrix},
\]
where $J_{n_1,n_2}$ is the $n_1 \times n_2$ all-ones matrix. By construction, $A \in S(G_1 \vee G_2)$.

Since $A_1$ is nonsingular, the Schur complement formula yields
\[
\det(A) = \det(A_1)\,\det\Big(A_2 - J_{n_2,n_1} A_1^{-1} J_{n_1,n_2}\Big).
\]

Note that $J_{n_2,n_1} = \mathbf{1}_{n_2}\mathbf{1}_{n_1}^{\top}$. Using  
$\mathbf{1}_{n_1}^{\top} A_1^{-1} \mathbf{1}_{n_1} = 0$, we obtain
\[
J_{n_2,n_1} A_1^{-1} J_{n_1,n_2}
= \mathbf{1}_{n_2} \Big(\mathbf{1}_{n_1}^{\top} A_1^{-1} \mathbf{1}_{n_1}\Big) \mathbf{1}_{n_2}^{\top}
= \mathbf{0}_{n_2},
\]
 where $\mathbf{0}_{n_2}$ is the zero matrix of order $n_2$.
Hence \(\det(A)=\det(A_1)\det(A_2)\neq 0\), so \(A\) is nonsingular.

Let $t \in \{1,2,\dots,n_1\}$. Then
\[
A(t)= 
\begin{pmatrix}
A_1(t) & J_{n_1-1,n_2}\\[2mm]
J_{n_2,n_1-1} & A_2
\end{pmatrix}.
\]
Applying the same argument to the block matrix $A(t)$,
\[
\det(A(t))=\det(A_2)\det(A_1(t)).
\]
Since $\det(A_1(t))=0$ for all $t \in \{1,2,\dots,n_1\}$, it follows that
\[
\det(A(t))=0 \quad \text{for all } t \in \{1,2,\dots,n_1\}.
\]
Similarly, for $t=n_1+j$ with $j\in\{1,\dots,n_2\}$,
\[
\det(A(t))=\det(A_1)\det(A_2(j))=0.
\]
Therefore, $A$ is a nonsingular matrix in $S(G_1\vee G_2)$ satisfying
\(
\det(A(t))=0
\ \text{for all }t.
\)
Hence $G_1\vee G_2$ has property~(P).
\end{pf}

\begin{rem}
The assumption that each graph has order at least $3$ in Theorem~\ref{join} is sufficient for the proof, but not necessary for the conclusion. Indeed, the assumption is used only to guarantee the existence of a zero-sum inverse realization of property~(P). Although $K_2$ has property~(P), it does not admit such a realization (Proposition~\ref{K_2}). Nevertheless,
\( K_2\vee K_2=K_4
\)
still has property~(P) (as shown in Theorem~\ref{thm:Kn_PropertyP}).
\end{rem}

\section{Property (P) for $H$-join of graphs}\label{sec:H-join}

Theorem~\ref{join} naturally extends to the more general $H$-join of a family of graphs. In this section, we establish this extension using the zero-sum inverse realization developed earlier.
\begin{defn}[$H$-join of graphs {\cite[p.~737]{cardoso}}]
Let $H$ be a graph with vertex set $\{1,2,\dots,k\}$, and let $\mathcal{F} =\{G_1, G_2, \dots, G_k\}$ be a family of $k$ pairwise disjoint graphs. The $H$-join of $\mathcal{F}$, denoted by 
\(
H[G_1,\dots,G_k]
\)
is the graph obtained from the disjoint union of $G_1,G_2,\dots,G_k$ by joining every vertex of $G_i$ to every vertex of $G_j$ whenever $ij\in E(H)$.
\end{defn}

\begin{thm}\label{thm:H-join}
Let $H$ be a graph with vertex set $\{1,2,\dots,k\}$, and let $\mathcal{F}=\{G_1, G_2, \dots, G_k\}$ be a family of pairwise disjoint graphs, each of order at least $3$. If every graph $G_i$ has property (P), then the $H$-join $H[G_1,\dots,G_k]$ has property (P).
\end{thm}
\begin{pf}
Let $G = H[G_1,\dots,G_k]$, and let $n_i = |V(G_i)|$ for each $i = 1, \dots, k$. Since each $G_i$ has property~(P) and is of order at least $3$,
Theorem~\ref{general} implies that each graph in the family admits a
zero-sum inverse realization of property~(P).
Hence, for each $i=1,\dots,k$, there exists a nonsingular matrix
\(
A_i \in S(G_i)
\)
such that
\(
\det A_i(t)=0 \text{ for all } t=1,\dots,n_i,
\)
and its inverse satisfies the zero-sum condition
\(
\mathbf{1}_{n_i}^\top A_i^{-1}\mathbf{1}_{n_i}=0.
\)

Define the block matrix
\[
A =
\begin{pmatrix}
A_1 & X_{12} & \cdots & X_{1k} \\
X_{21} & A_2 & \cdots & X_{2k} \\
\vdots & \vdots & \ddots & \vdots \\
X_{k1} & X_{k2} & \cdots & A_k
\end{pmatrix},
\]
where each off-diagonal block
\[
X_{ij} =
\begin{cases}
J_{n_i,n_j}, & \text{if } ij \in E(H),\\
\mathbf{0}_{n_i,n_j}, & \text{otherwise},
\end{cases}
\]
and $J_{n_i,n_j}$ and $\mathbf{0}_{n_i,n_j}$ denote the all-ones and zero matrices of order $n_i\times n_j$, respectively. By construction, $A$ is symmetric and belongs to $S(G)$.

Let
\(
N=\sum_{i=1}^k n_i=|V(G)|
\) 
and define
\[
D=\operatorname{diag}(A_1,\dots,A_k),
\]
together with the matrix $U$ of order \(N\times k\):
\[
U=
\begin{pmatrix}
\mathbf{1}_{n_1} & \mathbf{0} & \cdots & \mathbf{0} \\
\mathbf{0} & \mathbf{1}_{n_2} & \cdots & \mathbf{0} \\
\vdots & \vdots & \ddots & \vdots \\
\mathbf{0} & \mathbf{0} & \cdots & \mathbf{1}_{n_k}
\end{pmatrix},
\]
where each occurrence of $\mathbf{0}$ denotes a zero column vector of the appropriate dimension.

Let $A_H$ denote the adjacency matrix of $H$. Partition the matrix \(UA_HU^\top\) into \(k\times k\) blocks according to the block structure of \(A\). Then its \((i,j)\)-th block is
\[
(A_H)_{ij}\,\mathbf{1}_{n_i}\mathbf{1}_{n_j}^\top
=
\begin{cases}
J_{n_i,n_j}, & \text{if } ij\in E(H),\\
\mathbf{0}_{n_i,n_j}, & \text{otherwise}.
\end{cases}
\]
Since \(H\) is a simple graph, \((A_H)_{ii}=0\) for every
\(i=1,\dots,k\). Hence, every diagonal block of
\(UA_HU^\top\) is the zero matrix and, therefore
\[
A=D+UA_HU^\top.
\]

Since each $A_i$ is nonsingular, the block diagonal matrix $D$ is nonsingular. Moreover,
\[
D^{-1}=
\begin{pmatrix}
A_1^{-1} & \mathbf{0} & \cdots & \mathbf{0} \\
\mathbf{0} & A_2^{-1} & \cdots & \mathbf{0} \\
\vdots & \vdots & \ddots & \vdots \\
\mathbf{0} & \mathbf{0} & \cdots & A_k^{-1}
\end{pmatrix},
\]
where each occurrence of $\mathbf{0}$ denotes a zero matrix of the appropriate size.

Hence,
\begin{equation}\label{eq:detA formula}
A=D(I_N+D^{-1}UA_HU^\top).
\end{equation}
Therefore,
\begin{equation}\label{eq:detA=detD}
\begin{aligned}
\det(A)
&=\det(D)\det(I_N+D^{-1}UA_HU^\top)\\
&=\det(D)\det(I_k+A_HU^\top D^{-1}U),
\end{aligned}
\end{equation}
where the second equality follows from the well-known determinant identity
\cite[Corollary~18.1.2]{harville} (see also \cite[Theorem~1.3.22]{inter}),
applied to the matrices
\(
D^{-1}U\in\mathbb{R}^{N\times k}
\ \text{and}\ 
A_HU^\top\in\mathbb{R}^{k\times N}.
\)
Furthermore, \[(U^\top D^{-1} U)_{ij}
=
\begin{cases}
\mathbf{1}_{n_i}^\top A_i^{-1} \mathbf{1}_{n_i}, & \text{if } i=j,\\
0, & \text{otherwise.}
\end{cases}\]
Since $\mathbf{1}_{n_i}^\top A_i^{-1} \mathbf{1}_{n_i}=0$ for all $i$, it follows that
\(
U^\top D^{-1} U = \mathbf{0}_k.
\)
Therefore, from \eqref{eq:detA=detD}, 
\(
\det(A) = \det(D)\ne 0.
\)

Next, we show that $\det(A(t))=0$ for all $t \in V(G)$. 
Since $A$ is nonsingular, it is enough to show that every diagonal entry of
$A^{-1}$ is zero.

Since $U^\top D^{-1}U = \mathbf{0}_k$, we obtain
\[
(D^{-1} U A_H U^\top)^2 =
D^{-1}UA_H(U^\top D^{-1}U)A_HU^\top = \mathbf{0}_N.
\]
From \eqref{eq:detA formula}, it follows that
\[
A^{-1} = D^{-1} - D^{-1} U A_H U^\top D^{-1}.
\]
 The matrix $D^{-1}$ is block diagonal, with its $(p,p)$-th block equal to $A_p^{-1}$ and all off-diagonal blocks equal to the zero matrix.
 Moreover, under the same conformal block partition, the $(p,q)$-th block of
$D^{-1}UA_HU^\top D^{-1}$ is
\[
(A_H)_{pq}\,A_p^{-1}\mathbf{1}_{n_p}\mathbf{1}_{n_q}^\top A_q^{-1}.
\]

Since \(A_H\) is the adjacency matrix of a simple graph,
\(
(A_H)_{pp}=0
\ \text{for all } p=1,\dots,k.
\)
Therefore, the $(p,p)$-block of $A^{-1}$ is precisely
$A_p^{-1}$.

As each \(A_p^{-1}\) has zero diagonal entries, every diagonal entry of
\(A^{-1}\) is zero. Therefore,
\(
\det(A(t))=0
\ \text{for all } t\in V(G).
\)
Thus, $A$ is a nonsingular matrix in $S(G)$ satisfying
\(
P_\nu(A)=|V(G)|,
\)
and therefore
\(
H[G_1,\dots,G_k]
\)
has property~(P).
\end{pf}
Before applying Theorem~\ref{thm:H-join}, we establish that every complete graph possesses property~(P). Since we were unable to locate this result in the existing literature, we include a short proof for completeness.

\begin{thm}\label{thm:Kn_PropertyP}
The complete graph $K_n$ has property (P) for all $n > 1$.
\end{thm}
\begin{pf}
Consider the $n \times n$ matrix $$A = (1-n)I_n + J_n,$$ where $I_n$ is the identity matrix and $J_n$ is the all-ones matrix. Since all off-diagonal entries of $A$ are nonzero, $A \in S(K_n)$.

Recall that, for any positive integer $k$,
\[
\det(xI_k+J_k)=(x+k)x^{k-1}.
\]
Applying this formula with \(x=1-n\) and \(k=n\), we obtain
\[
\det(A)
=(1-n+n)(1-n)^{n-1}
=(1-n)^{\,n-1}.
\]
Since \(n>1\), it follows that \(\det(A)\neq0\). Hence \(A\) is nonsingular.

For each vertex $i$, the principal submatrix $A(i)$ is
\[
A(i)=(1-n)I_{n-1}+J_{n-1}.
\]
Again using the determinant formula with $x=1-n$ and $k=n-1$, we obtain
\[
\det(A(i))
=(1-n+n-1)(1-n)^{\,n-2}
=0.
\]
Thus every principal submatrix of order $n-1$ is singular. Therefore, $K_n$ has property~(P).
\end{pf}

Several well-known graph families possess property~(P), including complete graphs, bipartite graphs with perfect matchings, and cycles \cite{fons1,Howlader,sharma}. Theorem~\ref{thm:H-join} therefore provides a general mechanism for constructing new graphs with property~(P) from these families. We illustrate this principle through several applications.


Recall that a graph is called a \emph{cograph} (or \emph{complement-reducible graph}) if it contains no induced subgraph isomorphic to \(P_4\). The term \emph{cograph} reflects the fact that every such graph can be reduced to the empty graph by repeatedly taking complements of connected components \cite[p.~344]{west}. Not every cograph has property~(P); for example, \(K_1\) and \(P_3\) do not, while \(K_2\) does. This motivates the search for sufficient conditions under which a cograph has property~(P). Observe that the \(H\)-join of a family of cographs is again a cograph
whenever the index graph \(H\) is a cograph. Combining this observation
with Theorem~\ref{thm:H-join} yields the following corollary.

\begin{cor}\label{cor:cograph-P}
Let \(G\cong H[G_1,\dots,G_k]\), where \(H\) and \(G_1,\dots,G_k\) are cographs. If each \(G_i\) has order at least \(3\) and possesses property~(P), then \(G\) has property~(P).
\end{cor}

\begin{pf}
Since \(H\) and \(G_1,\dots,G_k\) are cographs, the graph
\(
H[G_1,\dots,G_k]
\)
is a cograph. The conclusion now follows immediately from Theorem~\ref{thm:H-join}.
\end{pf}

As a first consequence, choose each constituent graph to be either a complete graph or a  complete balanced bipartite graph. Both graph families are cographs and possess property~(P).

Consequently, we obtain the following family of complete multipartite graphs
\[
K_{n_1,n_1,\ldots,n_t,n_t}, \ n_i\ge2.
\]
Since
\[
K_{n_1,n_1,\ldots,n_t,n_t}
=
K_t\!\left[
K_{n_1,n_1},
K_{n_2,n_2},
\ldots,
K_{n_t,n_t}
\right],
\]
where \(K_t\) and each \(K_{n_i,n_i}\) are cographs, and every \(K_{n_i,n_i}\) has order at least \(4\) and has property~(P), Corollary~\ref{cor:cograph-P} yields the following result.

\begin{cor}\label{cor:complete-multipartite}
For all integers $n_1,n_2,\ldots,n_t\geq 2$, the complete multipartite graph
$$
K_{n_1,n_1,\ldots,n_t,n_t}
$$
has property~(P).
\end{cor}

\begin{rem}
Recall that the Tur\'an graph \(T_r(N)\) is the complete \(r\)-partite
graph on \(N\) vertices whose partite sets differ in size by at most
one \cite[p.~181]{diestel}. In particular, for \(t\ge1\) and
\(q\ge2\),
\[
T_{2t}(2tq)
\cong
K_{\underbrace{q,q,\ldots,q}_{2t\text{ parts}}}.
\]
Hence, by Corollary~\ref{cor:complete-multipartite}, every balanced
Tur\'an graph with an even number of parts, each of size at least
\(2\), possesses property~\textup{(P)}.
Setting \(q=2\), we obtain
\[
T_{2t}(4t)
\cong
K_{\underbrace{2,2,\ldots,2}_{2t\text{ parts}}},
\]
which is the cocktail party graph on \(4t\) vertices
\cite[p.~68, Exercise~9(b)]{biggs}. Hence, every cocktail party graph
on \(4t\) vertices possesses property~\textup{(P)}.
\end{rem}

Our final application concerns lexicographic products, which arise as a special case of the $H$-join.
Following the notation of \cite[p.~393]{west}, we denote the lexicographic product by $H[G]$. It is obtained by replacing each vertex of $H$ with a copy of $G$. Equivalently, if $|V(H)|=k$, then
$$
H[G] = H[\underbrace{G,\ldots,G}_{k\text{ copies}}].
$$
Hence, the lexicographic product is a special case of the $H$-join, and the following corollary follows from Theorem~\ref{thm:H-join}.

\begin{cor}\label{cor:lexicographic-product}
Let $G$ be a graph of order at least $3$ with property~(P). Then, for every graph $H$, the lexicographic product $H[G]$ has property~(P).
\end{cor}

\begin{example}
Since cycles and complete graphs possess property~(P), it follows from Corollary~\ref{cor:lexicographic-product} that, for every graph \(H\), the lexicographic products \(H[C_n]\) and \(H[K_n]\) possess property~(P) for all \(n \ge 3\).
\end{example}
\section*{Declaration of competing interest}
There is no competing interest.
 
\section*{Acknowledgements}

The research of the first author was carried out with support from  
the ANRF Startup Research Grant SRG/2022/001281, the ANRF Matrics Grant ANRF/ARGM/2025/002777/MTR and RG26271235MARGFX009003 of IIT Madras.
 The research of the second author was supported by the Prime Minister's Research Fellowship (PMRF, ID: 2503482). The research of the third author was supported by the National Board for Higher Mathematics (NBHM) Ph.D. Scholarship (0203/8(28)/2023-R\&D-II), Govt. of India.

\end{document}